\documentclass[12pt,a4paper] {article}

\usepackage[utf8]{inputenc}
\usepackage{amsthm,amsmath,amssymb}
\usepackage{tikz}
\usepackage{graphicx,subcaption,caption}
\usepackage{algorithm,algorithmic}
\usepackage{pgfplots}
\pgfplotsset{width=7cm,compat=1.3}
\usepgfplotslibrary{external}
\newtheorem{theorem}{Theorem}
 
\newtheorem{lemma}{Lemma}

\newtheorem{remark}{Remark}
\newtheorem{corollary}{Corollary}
\newtheorem{definition}{Definition}
\usepackage{hyperref}
\providecommand{\keywords}[1]{\textbf{\textit{Keywords---}} #1}

\begin{document}
\title{On convergence to the global optimum}
\author{K. Lakshmanan \thanks{Department of Computer Science and Engineering, Indian Institute of Technology (BHU), Varanasi 221005, Email: lakshmanank.cse@iitbhu.ac.in }}
\date{}

\maketitle
\begin{abstract}
	We show that there is no general algorithm which finds a sequence of points with finite-precision converging computably to the global optima of any continuous function.
\end{abstract}


%\ccode{Mathematics Subject Classification 2020: 03D55}

\section{Introduction and Preliminaries}
We consider the problem of finding the global minima of a non-convex continuous function $f:  C \rightarrow \mathbb{R}$, where $ C \subset \mathbb{R}^d$ is a closed, compact subset. Global minima is the point $x^* \in C$ which satisfies the following property: $f(x^*) \leq f(x)$ for all $x \in C$. The function $f$ attains this minimum at least once by extreme value theorem. Our goal is to find one such point. This problem is well-studied with many books written on the subject, see for example \cite{globook}.

We note that our setting  is different from the computability of similar problems studied for example in \cite{pour}. In these works they study computable real numbers and computable real functions. And it is known that if there is an isolated maximum, the point and the maximum value are computable \cite{pour}. In our case, we assume that the function is known by it's oracle. We consider finite-precision numbers and the general case of any continuous function.
Since computers have limited memory, oracles with real-numbers are not a very useful model. Hence we  consider only finite-precision reals.

While it is easy to see that a simple grid search will output a sequence of points converging to the global optima. And for a Lipschitz continuous function it requires exponential number of oracle calls \cite{nest}. 
%It can be seen that this convergence is not computable. 
Our main result is that for any continuous function there
is no algorithm that computes a sequence of points (finite-precision) converging computably to the global minima.

\subsection{Finite Precision Reals}
We now briefly explain what we mean by this.
Consider any real number $x \in \mathbb{R}$. Let $n_0$ be the largest integer such that $r_0 \leq x$. Having chosen $r_0,r_1,\ldots,r_{k-1}$ choose largest positive integer $r_k$ such that
\[ r_0 + \frac{r_1}{10} + \frac{r_2}{10^2} + \ldots \frac{r_k}{10^k} \leq x. \]

This is the decimal expansion of the number. We can check that this expansion is unique. We define precision length to be the number $k$.
Now for finite precision representaion we need to specify the precision length $k$. For any real $x \in \mathbb{R}$, the numbers $r_0,r_1,\ldots,r_k$ is its finite precision representation. Note that $r_i, 0\leq i \leq k$ can be zero.  %We can consider this finite-precision real number as a natural number by representing it as $r_0*10^k+r_1*10^{k-1}+\ldots+r_k$. For $x \in \mathbb{R}^d$,%all finite-precision lengths of its coordinates.
For a point $x \in \mathbb{R}^d$, given a precision length $k$ we can have decimal expansions for all it's co-ordinates.
Note that though we give binary representations to the Turing machine, for simplicity we assume precisions denote the decimal precisions.

\begin{remark}\label{gaprem}
	Suppose $r_1,\ldots,r_k$ is the finite precision representation with length $k$ of some real $x$. Let $\bar x$ be the number with decimal expansion $r_1,\ldots,r_k$ as $x$ and $r_{l}=9$ for $l \geq k+1$. And let $\underline x$ be the number with decimal expansion $r_1,\ldots,r_k$ as $x$ and $r_{l}=0$ for $l \geq k+1$. We can see that for all reals in $[\underline x,\bar x]$ we have the same finite length representation. And the length of this interval is $\epsilon = 10^{-k}$. We then say with precision length $k$ we can represent consecutive numbers with gap greater than or equal to $\epsilon$.
\end{remark}

\subsection{Computable Convergence}
Our result is about computable convergence. We define it here.
\begin{definition}
	We say that a sequence of computable numbers ${x_n}$ converges computably to $x^*$, if for all $\epsilon > 0$ there is integer-valued computable function $N(\epsilon)$ such that if $\epsilon > 0$ and $n > N(\epsilon)$ then $|x_n - x^*| < \epsilon$.
\end{definition}

\subsection{The Problem}
We assume there is an oracle for our continuous function $f$. This oracle gives the value $f(x)$ upto any finite-precision for an given finite-precision $x$.
The Turing-machine has access to this function oracle. We give also give a value $\delta > 0$ as input to the Turing machine. It needs to write the any point $x_{o}$ of finite precision length such that $|f(x_{o})-f(x*)| < \delta$ i.e., it should find $\delta-$ approximation of the global optima. We show that this problem is not computable.

Let us assume we have a three-tape Turing machine, one is used for calculation, one is for the giving the finite precision real and the precision required to the function oracle and one has the value returned from the oracle. \cite{sorbook}

\begin{definition}
	Turing machine has a three infinite tapes divided into cells, a reading head which scans one cell of the tape at a time, and a finite set of internal states $Q=\{q_0,q_1,\ldots,q_n\}, \, n \geq 1$. Each cell is either blank or has symbol 1 written on it. In a single step the machine may simultaneously (1) change the from one state to another; (2) change the scanned symbol $s$ to another symbol $s'\in S = \{1,B\}$; (3) move the reading head one cell to the right (R) or left (L). This operation of machine is controlled by a partial map $\Gamma : Q \times S^3 \rightarrow Q \times (S \times \{R,L\})^3$.
\end{definition}

The map $\Gamma$ viewed as a finite set of quintuples is called a Turing program. The interpretation is that if $(q,s_1,s_2,s_3,q',s'_1,X_1,s'_2,X_2,s'_3,X_3) \in \Gamma$, in state $q$, scanning symbols $s_1, s_2, s_3$ changes state to $q'$ and in the tape $i$ input symbol to $s'_i$ and moves to scan one square to the right if $X_i=R$ (or left if $X_i=L$.) in the tape $i$.

\section{Main Theorem}

%Given a input $x$ of finite precision length $n$, the oracle can return $f(x)$ to any specified precision.
 Given the function $f$, let the set of global minima be denoted by $G_f$. Now consider $\delta$-approximation to the global minima.  %We consider finite length approximation of the set $G'_{\delta}$. 

\begin{lemma}\label{preclem}
	For all $\delta > 0$ there exists a point $x_n^*$ of finite precision length $n$ such that $|f(x^*)-f(x_n^*)| < \delta$. 
\end{lemma}

\begin{proof}
	Let $\epsilon > 0$ be such that $|x-y| < \epsilon$ implies $|f(x)-f(y)| < \delta$. Such an $\epsilon > 0$ exists for all $\delta > 0$ because the function $f$ is continuous. Let $n$ be the precision length required to represent numbers with gap $\epsilon/10$ between consecutive numbers (Remark \ref{gaprem}). Then we see for the global minima $x^*$ (like for all other points) it's finite precision representation $x_n^*$ with precision length $n$ is such that $|x_n^*-x^*|<\epsilon$. 
\end{proof}

\begin{definition}
	Let $G'_{\delta,k}$ be the set of points with given finite precision length $k \geq 1$ where the function value is $\delta > 0$ close to the global minima. And $G'_{\delta}$ be the union of all such sets.
\end{definition} % We show that this set $G'_{\delta}$ is not limit computable.
We consider only finite-precision numbers. %As the domain is a compact set, this can be regarded as a subset of natural numbers by multiplying by an appropriate constant. And in particular
As there are only finite number of points with precision length $k$, the set $G'_{\delta,k}$ is finite. 
Since we would like an algorithm to converge to a single point, for simplicity we assume the global optima is unique i.e., $G_f$ is a singleton. We can state the main theorem.

\begin{theorem}\label{mainthm}
	There is no sequence of computable points with finite-precision length converging computably to the global minima.
\end{theorem}

\begin{proof}
	Let assume we have a sequence of $\{\delta_k\}$ which goes to zero. Let $x_k$ be any point of
	some finite precision length $n_k$ such that $|f(x^*)-f(x_k)| < \delta_k$. Such a point exists by Lemma \ref{preclem} i.e., the set $G'_{\delta,n_k}$ is non-empty for $\delta > 0$. This length $n_k$ can increase with $k$.
	
	We consider an equivalent problem. We define $h_{x_k}^{\delta}(x) := \max \{0,f(x_k)-f(x)-\delta\}$. Since our objective function $f$ is continuous, $h_{x_k}^{\delta}(\cdot)$ is also continuous. 
	This function is identically zero if and only if $|f(x_k) - f(x)| < \delta$, for all $x$. This happens only if $x_k$ is a $\delta-$ approximation to the optimum. Note that $x_k$ and $x$ are represented with some finite precision. % then $h_{x_k}^{\delta}(\cdot)$ is identically zero. %Or to find a finite-precision point $y$ which is $\delta-$ approximation to the optimum is same as finding a function $h_y^{\delta}(\cdot)$ which is identically zero for some $y$. But this cannot be checked for any particular value of $y$. 
	
	Thus the set of all points with finite precision that are $\delta$ close to the global minima $G'_{\delta,k}$, is also the set of all points $x_k$ where the function $h_{x_k}^{\delta}(\cdot)$ is identically zero.
	That is to find if a point $x_k$ belongs to $G'_{\delta,n_k}$ to is same as checking if whether $h_{x_k}^{\delta}(x) \equiv 0 $ (function is identically zero). But this cannot be checked for a particular $x_k$ unless it is checked for all $x$ of any finite precision length. As there are infinitely many such points, there is no Turing machine which can compute (halt) if a function is zero at infinitely many points.
	Hence finding a $\delta-$ approximation $\{x_k\}$ with precision length $n_k$ to the optimum is not computable %is belongs to $G'_{\delta,n_k}$ is not computable
	or in other words we have shown the theorem.
	%there is no sequence of computable $\{x_k\}$ with finite-precision length converging to the global minima.
\end{proof}

\begin{corollary}
	The problem of approximating the global minima of a continuous function by an arbitrary $\delta > 0$ is not computable.
\end{corollary}

\begin{corollary}
	The problem of checking whether local minima $z$ is global is not computable as this involves checking whether $h_z^{\delta}(\cdot)$ is identically zero.
\end{corollary}

\begin{remark}
	As we can find an ball of some radius $\omega$ where the continuous function $h_z^{\delta}(\cdot)$ is non-zero around a local optima. The same proof does not hold for converging to local optima as we can check if $h_z^{\delta}(\cdot)$ is identically zero in steps of size less than $\omega$.
\end{remark}

\begin{remark}
	Even in presence of higher order oracles, i.,e orcales which give derivates of the function, the same problem of checking if $h_z^{\delta}(\cdot)$ is identically zero remains. Hence global optima even in presence of these higher-order function oracle is not computable.
\end{remark}

\section{Conclusion}
We have given a simple proof that there is no algorithm which finds a sequence (of finite-precision) points converging computably to global optima of any continuous function $f$. %To the best of our knowledge, this is the first such result.

\bibliographystyle{plain}

\begin{thebibliography}{5}
	\bibitem{globook} R. Horst and H. Tuy., Global Optimization: Deterministic Approaches, Springer-Verlag (1996).
	
	\bibitem{sorbook} R.I. Soare., Turing Computability: Theory and Applications, Springer-Verlag (2016).
	
	\bibitem{nest} Y. Nesterov., Introductory Lectures on Convex Optimization A Basic Course, Kluwer Academic Publishers (2003).
	
	\bibitem{pour} M. Pour-El and J. Richards., Computability in analysis and physics, Springer, Heidelberg (1989).
\end{thebibliography}

\end{document}